\documentclass[12pt]{amsart}
\usepackage{amsmath,amsthm,amsfonts,amssymb,latexsym,verbatim,bbm,longtable}
\input xy
\xyoption{all}
\usepackage{hyperref,multirow}
\usepackage{tikz, tabularx,booktabs,float}
\usepackage[shortlabels]{enumitem}
\usepackage{young}

\allowdisplaybreaks[2] \textwidth15.1cm \textheight22cm \headheight12pt \oddsidemargin.4cm
\theoremstyle{plain}
\newtheorem{theorem}{Theorem}[section]
\newtheorem{thm}[theorem]{Theorem}
\newtheorem{lemma}[theorem]{Lemma}
\newtheorem{proposition}[theorem]{Proposition}

\newtheorem{definition}[theorem]{Definition}

\newtheorem{cor}[theorem]{Corollary}

\newcommand{\K}{\mathbb{K}}

\DeclareMathOperator{\mfS}{\mathfrak{S}}

\DeclareMathOperator{\Gr}{Gr}

\DeclareMathOperator{\Fl}{F\ell}

\newcommand{\bull}{\bullet}

\DeclareMathOperator{\Span}{span}
\DeclareMathOperator{\Left}{\textbf{Left}}
\DeclareMathOperator{\Right}{\textbf{Right}}
\begin{document}

\title[Split Pattern Avoidance]{Enumerating permutations avoiding split patterns $3|12$ and $23|1$}

\author{Travis Grigsby}
\email{travimg@okstate.edu}

\author{Edward Richmond}
\email{edward.richmond@okstate.edu}

\maketitle
\begin{abstract}

In this paper, we give a formula for the number of permutations that avoid the split patterns $3|12$ and $23|1$ with respect to a position $r$.  Such permutations count the number of Schubert varieties for which the projection map from the flag variety to a Grassmannian induces a fiber bundle structure.  We also study the corresponding bivariate generating function and show how it is related to modified Bessel functions.
\end{abstract}

\section{Introduction}

Let $\mfS_n$ denote the permutation group on the set of integers $\{1,2,\ldots, n\}$.  In \cite{AR18}, the second author and Alland define the notion of split pattern avoidance on permutations and use this idea to combiantorially characterize when the projection map of a Schubert variety to a Grassmannian gives rise to a fiber-bundle structure.  Split pattern avoidance is a specialization of classical pattern avoidance on permutations and can be described as follows.  A \textbf{split pattern} $u=u_1|u_2$ is a divided permutation where $u_1=u(1)\cdots u(j)$ and $u_2=u(j+1)\cdots u(n)$ for some $1\leq j\leq n$.

\begin{definition}\cite[Definition 2.1]{AR18}
Let $w\in\mfS_n$ and $u\in \mfS_k$.   We say $w=w(1)\cdots w(n)$ \textbf{contains the split pattern} $u=u(1)\cdots u(j)|u(j+1)\cdots u(k)$ \textbf{with respect to position} $r$ if there exists a sequence $(1\leq i_1<\cdots<i_k\leq n)$ such that
\begin{enumerate}
\item $w(i_1)\cdots w(i_k)$ has the same relative order as $u$ and
\item $i_j\leq r<i_{j+1}$.
\end{enumerate}
If $w$ does not contain $u$ with respect to position $r$, then we say $w$ \textbf{avoids the split pattern} $u$ \textbf{with respect to position} $r$.
\end{definition}

Note that condition (1) in the above definition is the classical definition of pattern containment.  If a permutation $w$ avoids a split pattern $u$ with respect to all positions $r\leq n$, then $w$ avoids $u$ in the traditional sense.  In this paper, we focus on permutations that avoid the split patterns $3|12$ and $23|1$ which are the same split patterns studied in \cite{AR18} in connection with Schubert varieties.  Diagrammatically, we represent a permutation $w$ with its corresponding the permutation matrix  with nodes marking the points $(i,w(i))$.  Here the point $(1,1)$ marks the lower left corner.  For example, if we consider $w=315642$,
then the corresponding permutation matrix is given by:

$$
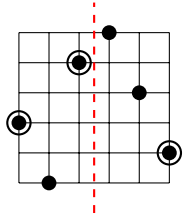
\begin{tikzpicture}[scale=0.4]
\draw[step=1.0,black] (0,0) grid (5,5);
\fill (0,2) circle (7pt);\draw[thick] (0,2) circle (11pt);
\fill (1,0) circle (7pt);
\fill (2,4) circle (7pt);\draw[thick] (2,4) circle (11pt);
\fill (3,5) circle (7pt);%\draw[thick] (3,0) circle (11pt);
\fill (4,3) circle (7pt);
\fill (5,1) circle (7pt);\draw[thick] (5,1) circle (11pt);
\draw[dashed,thick, red] (2.5,-1)--(2.5,6);
%\draw[dashed,thick, blue] (-1,2.5)--(9,2.5);
\end{tikzpicture}$$

We can see that $w$ contains $23|1$ with respect to position $3$ (the pattern containment is marked by the circled nodes).  We can also see that $w$ avoids $3|12$ with respect to position $3$ since all the nodes to the right of position 3 are decreasing.

\begin{comment}
The patterns $3|12$ and $23|1$ are represented by
 $$\begin{tikzpicture}[scale=0.45]
\draw[step=1.0,black] (0,0) grid (2,2);
\fill (0,2) circle (7pt);%\draw[thick] (0,2) circle (11pt);
\fill (1,0) circle (7pt);
\fill (2,1) circle (7pt);%\draw[thick] (2,0) circle (11pt);
\draw[dashed,thick, red] (0.5,-1)--(0.5,3);
%\draw[dashed,thick, blue] (-1,2.5)--(9,2.5);
\end{tikzpicture}\qquad\text{and}\qquad
\begin{tikzpicture}[scale=0.45]
\draw[step=1.0,black] (0,0) grid (2,2);
\fill (0,1) circle (7pt);%\draw[thick] (0,2) circle (11pt);
\fill (1,2) circle (7pt);
\fill (2,0) circle (7pt);%\draw[thick] (2,0) circle (11pt);
\draw[dashed,thick, red] (1.5,-1)--(1.5,3);
%\draw[dashed,thick, blue] (-1,2.5)--(9,2.5);
\end{tikzpicture}$$
and permutations avoiding these patterns can be visualized by their corresponding matrix avoiding nodes in these relative patterns.
\end{comment}
Let $K(r,n)$ denote the set of permutations in $\mfS_n$ that avoid the split patterns $3|12$ and $23|1$ with respect to position $r$ and define \[k(r,n):=|K(r,n)|.\] For convenience in enumeration, we set $k(0,0):=1$.  Observe that rotation of a permutation matrix by 180 degrees gives a bijection between the sets $K(r,n)$ and $K(n-r,n)$ and hence $k(r,n)=k(n-r,n)$. For any integer $m$ and $i\geq 1$, let \[(m)_i:=m(m-1)\cdots (m-i+1)\] denote the falling factorial.
The first main result of this paper the following formula for $k(r,n)$:

\begin{theorem}\label{T:main1}
For any, $0\leq r\leq n$, we have
\[k(r,n)=r!(n-r)!+\sum_{i=1}^{r}\sum_{j=1}^{n-r}\binom{n-i-j}{r-i}\cdot (r)_{i-1}\cdot(n-r)_{j-1}.\]
\end{theorem}

Table \ref{TBL:count} below lists values of $k(r,n)$ for $n\leq 9$ and $r\leq 4$.  Note that, by the symmetry $k(r,n)=k(n-r,n)$, all values of $k(r,n)$ for $n\leq 9$ can be deduced from the table.

\begin{table}[H]
\begin{tabular}{|c|ccccccccc|}
\hline
$r\backslash n$ & 1 & 2& 3 & 4& 5 &6 & 7 & 8 & 9\\ \hline
0 & 1 & 2 & 6 & 24 & 120 & 720 & 5040 & 40320 & 362880\\
1 & 1 & 2 & 5 & 16 & 65  & 326 & 1957 & 13700 & 109601\\
2 &   & 2 & 5 & 14 & 47  & 194 & 977  & 5870  & 41099\\
3 &   &   & 6 & 16 & 47  & 162 & 676  & 3416  & 20541\\
4 &   &   &   & 24 & 65  & 194 & 676  & 2836  & 14359\\ \hline
\end{tabular}
\medskip
\caption{Number of permutations in $\mfS_n$ avoiding the split patterns $3|12$ and $23|1$ with respect to position $r$.}
    \label{TBL:count}
\end{table}

Theorem \ref{T:main1} implies that the value $k(r,n)\geq r!(n-r)!$ and hence we consider the bivariate exponential generating function \[\mathcal{K}(x,y):=\sum_{n=0}^{\infty}\sum_{r=0}^n k(r,n) \frac{x^r y^{n-r}}{r!(n-r)!}.\]
The second main result of this paper is a formula for $\mathcal{K}(x,y)$ which has a surprising connection with Bessel functions.  Consider the complex valued function
\[I_0(z)=\frac{1}{\pi}\int_0^\pi e^{z\cos\theta}d\theta.\]
The function $I_0(z)$ is called the modified Bessel function of first kind (at $\alpha=0$) and is a solution to the modified Bessel differential equation
\[z^2u_{zz}+zu_z-z^2u=0.\]
Bessel equations and their solutions are important objects in the study of partial differential equations and have several applications in physics including the modeling of heat conduction, membrane vibration, hydrodynamics and electrostatics \cite[Chapter 7]{Tr69}.  For more on the general theory of Bessel functions see \cite{Wa44}.  In this paper, we will primarily use the fact that $I_0$ has power series expansion
\begin{equation}\label{D:Bessel_def}
I_0(z)=\sum_{m=0}^\infty \frac{1}{(m!)^2}\left(\frac{z^2}{4}\right)^m.
\end{equation}
In fact, this power series expansion of $I_0$ is commonly taken as the definition of the modified Bessel function.  The following theorem is the second main result of this paper.
\begin{theorem}\label{T:main2}
The generating function $\mathcal{K}(x,y)$ satisfies
\[\mathcal{K}(x,y)=\frac{\mathcal{L}(x,y)+1}{1-x-y+xy}\]
where
\begin{equation}\label{Eq:Bessel_intro}
\mathcal{L}(x,y)=\iint e^{x+y}I_0(2\sqrt{xy})\ dx\, dy
\end{equation}
with boundary conditions $\mathcal{L}(x,0)=e^x$ and $\mathcal{L}(0,y)=e^y.$
\end{theorem}
In Proposition \ref{P:binom_gen_fun}, we prove that the expression $e^{x+y}I_0(2\sqrt{xy})$ in the integral from Equation \eqref{Eq:Bessel_intro} is a bivariate exponential generating function for binomial coefficients. Specifically, we show that
\begin{equation}\label{Eq:Bessel_intro2}
\sum_{r,s=0}^{\infty} \binom{r+s}{r}\frac{x^{r}y^{s}}{r!\, s!}=e^{x+y}I_0(2\sqrt{xy}).
\end{equation}
We remark that there are two other commonly known bivariate generating functions for binomial coefficients given by
\[\sum_{r,s=0}^\infty \binom{r+s}{r} x^ry^s=\frac{1}{1-x-y}\quad\text{and}\quad\sum_{r,s=0}^\infty \binom{r+s}{r} \frac{x^ry^s}{(r+s)!}=e^{x+y}.\]
The authors are unaware of any previously known statement or proof for the ``intermediate" generating function for binomials given in Equation \eqref{Eq:Bessel_intro2}.

Observe that specializing to $x=y$ in Equation \eqref{Eq:Bessel_intro2} gives
\begin{align*}
e^{2x}\cdot I_0(2x)&=\sum_{r,s=0}^{\infty} \binom{r+s}{r}\frac{x^{r+s}}{r!\, s!}\\
&=\sum_{r,s=0}^{\infty} \binom{r+s}{r}^2\frac{x^{r+s}}{(r+s)!}.
\end{align*}
Setting $m=r+s$ yields
\begin{align*}
e^{2x}\cdot I_0(2x)=\sum_{m=0}^\infty\left(\sum_{r=0}^m \binom{m}{r}^2\right)\frac{x^{m}}{m!}=\sum_{m=0}^\infty\binom{2m}{m}\frac{x^{m}}{m!}.
\end{align*}
Hence we recover the formula for the exponential generating function for central binomial coefficients.  In particular, Equation \eqref{Eq:Bessel_intro2} can be viewed as a bivariate refinement of the generating function for central binomial coefficients with respect to Vandermonde's identity.  We note that the connection between central binomial coefficients and Bessel functions was previously known \cite{OEIS:A000984} and can be deduced directly from Equation \eqref{D:Bessel_def}.

\subsection{Connections with Schubert varieties}
In this subsection, we give an overview the results found in \cite{AR18} connecting split pattern avoidance with the fiber bundle structures of Schubert varieties.  These fiber bundle structures are closely connected to Billey-Postnikov decompositions which studies the parabolic factoring of lower Bruhat intervals in Coxeter groups \cite{BP05, RS16}.  We also remark that pattern avoidance is a tool commonly used to combinatorially characterize various geometric properties of Schubert varieties.  Most notably, Lakshmibai and Sandhya show that a Schubert variety is smooth if and only if its indexing permutation avoids the patterns 3412 and 4231 \cite{LS90}.  For a survey on the connections between the geometry of Schubert varieties and pattern avoidance see \cite{AB16}.
Let $\K$ be an algebraically closed field and let
\[\Fl(n):=\{(V_1\subset V_2\subset\cdots \subset V_{n-1}\subset\K^n)\ | \ \dim(V_i)=i\}\]
denote the complete flag variety on $\K^n.$   For each $1\leq r\leq n$, let \[\Gr(r,n):=\{V\subseteq\K^n\ |\ \dim(V)=r\}\] denote the Grassmannian of $r$-dimensional subspaces of $\K^n$ and consider the projection map
\[\pi_r:\Fl(n)\rightarrow \Gr(r,n)\]
given by $\pi_r(V_\bullet)=V_r.$  Fix a basis $\{e_1,\ldots,e_n\}$ of $\K^n$ and let $E_i:=\Span\langle e_1,\ldots,e_i\rangle$.  Each permutation $w\in \mfS_n$ defines a Schubert variety
$$X_w:=\{V_\bull\in\Fl(n)\ |\ \dim(E_i\cap V_j)\geq r_w[i,j]\}$$
where $r_w[i,j]:=\#\{k\leq j \ |\ w(k)\leq i\}$.  The following theorem gives a connection between split pattern avoidance and fiber bundle structures on $X_w$ with respect to the projection map $\pi_r$.
\begin{thm}\cite[Theorem 1.1]{AR18}
\label{T:Alland-Richmond}
Let $r\leq n$ and $w\in \mfS_n$. The projection $\pi_r$ restricted to $X_w$ is a Zariski-locally trivial fiber bundle if and only if $w$ avoids the split patterns $3|12$ and $23|1$ with respect to position $r$.
\end{thm}
Theorem \ref{T:Alland-Richmond} implies that the number $k(r,n)$ can alternately be viewed as the number of Schubert varieties $X_w\subseteq\Fl(n)$ for which the projection $\pi_r$ to the Grassmannian $\Gr(r,n)$ induces a fiber bundle structure on $X_w$.

For the remainder of this paper, we focus on proving Theorems \ref{T:main1} and \ref{T:main2}.

\section{Enumerating split patterns}

In this section we prove Theorem \ref{T:main1}.  For any $w\in \mfS_n$ and $0\leq r\leq n$, define the integer sets
\[\Left_r(w):=\{w(k)\ |\ k\leq r\}\quad\text{and}\quad \Right_r(w):=\{w(k)\ |\ k>r\}.\]
For example, if $w=7361254$.  Then $\Left_4(w)=\{1,3,6,7\}$ and $\Right_4(w)=\{2,4,5\}$.  Recall that $K(r,n)$ denotes the set of permutations in $\mfS_n$ that avoid the split patterns $3|12$ and $23|1$ with respect to position $r$ and define
\[K_L(r,n):=\{w\in K(r,n)\ |\ n\in \Left_r(w)\}\]
and
\[K_R(r,n):=\{w\in K(r,n)\ |\ n\in \Right_r(w)\}.\]

In other words, $K_L(r,n)$ consists of those permutations in $K(r,n)$ for which $n$ appears within the first $r$ positions in the one-line notation of $w$ and and $K_R(r,n)$ are those for which $n$ appears within the last $n-r$ positions. Clearly, we have
\[K(r,n)=K_L(r,n)\sqcup K_R(r,n).\]
Let $\phi_n:\mfS_n\rightarrow \mfS_{n-1}$ denote the map that removes $n$ from the one-line notation of a permutation.  For example, we have
\[\phi_6(432\textbf{6}15)=43215.\]
Note that if $w\in\mfS_n$ avoids the split patterns $3|12$ and $23|1$ with respect to position $r$, then $\phi_n(w)$ will avoid these patterns with respect to either position $r$ or $r-1$. In particular, if $w\in K_R(r,n)$, then $\phi_n(w) \in K(r,n-1)$, and if $w\in K_L(r,n)$, then $\phi_n(w) \in K(r-1,n-1)$.
The next proposition is on the restriction of $\phi_n$ to $K_R(r,n)$.

\begin{proposition}\label{Prop:projection_fibers}
The restricted map $\phi_n:K_R(r,n)\rightarrow K(r,n-1)$ is surjective.  Moreover the fiber size $|\phi_n^{-1}(w)|=n-r$ for any $w\in K(r,n-1)$.
\end{proposition}

\begin{proof}
Let $w\in K(r,n-1)$ and for any $1\leq i\leq n-r$, define the permutation $w_i\in \mfS_n$ by

\[w_i(k):=\begin{cases}w(k) & \text{if $k< r+i$}\\
n & \text{if $k=r+i$}\\
w(k-1) & \text{if $k>r+i$.}
\end{cases}\]
In other words, $w_i$ is the permutation in $S_n$ arising form inserting $n$ into the $(r+i)$-th position of the one-line notation of $w$.  It is easy to see that $\phi_n(w_i)=w$.  Since $w$ avoids $3|12$ and $23|1$ with respect to position $r$ and $n\in \Right_r(w_i)$, we have that $w_i$ also avoids $3|12$ and $23|1$ with respect to position $r$.  In particular, $w_i\in K_R(r,n)$.  This implies the restricted map $\phi_n:K_R(n,r)\rightarrow K(r,n-1)$ is surjective with $|\phi_n^{-1}(w)|\geq n-r$.  Finally, suppose $u\in K_R(r,n)$.  Then $u^{-1}(n)=r+j$ for some $1\leq j\leq n-r$ and $u=\phi_n(u)_j$.  This implies $|\phi_n^{-1}(w)|=n-r$ for any $w\in K(r,n-1)$.
\end{proof}

One consequence of Proposition \ref{Prop:projection_fibers} is the following corollary.

\begin{cor}\label{C:K2K_L}
For any $1\leq r\leq n$, we have
\[k(r,n)=\sum_{j=0}^{n-r} (n-r)_j\cdot |K_L(r,n-j)|.\]
\end{cor}

\begin{proof}
Since $K(r,n)=K_L(r,n)\sqcup K_R(r,n)$, we have \[k(r,n)=|K_R(r,n)|+|K_L(r,n)|.\]  Proposition \ref{Prop:projection_fibers} implies
$|K_R(r,n)|=(n-r)\cdot k(r,n-1)$ and hence
\[k(r,n)=(n-r)\cdot k(r,n-1)+|K_L(r,n)|.\]
Applying this decomposition again on $k(r,n-1)$ and iterating this process yields the desired formula.  Note the iteration terminates when $n=r$ in which case we get $|K_L(r,r)|=k(r,r)=r!$.
\end{proof}

To prove Theorem \ref{T:main1}, it suffices to calculate $|K_L(r,n)|$.

\begin{proposition}\label{P:K_L_formula}
For any $1\leq r< n$, we have
\[|K_L(r,n)|=\sum_{i=1}^{r}\binom{n-i-1}{r-i}\cdot (r)_{i-1}.\]
\end{proposition}

\begin{proof}
Fix $1\leq r< n$ and for any $1\leq i\leq n$, define the set
\[S(i):=\{w\in K_L(r,n)\ |\ \text{$i\in\Right_r(w)$ and $k\in\Left_r(w)$ for all $k<i$}\}.\]
For example, if $n=9$, $r=6$ and $i=4$, then $w=391276|854\in S(4)$.

$$
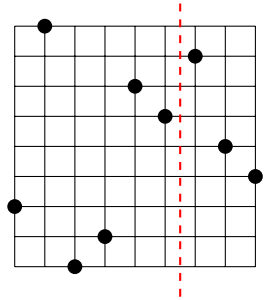
\begin{tikzpicture}[scale=0.4]
\draw[step=1.0,black] (0,0) grid (8,8);
\fill (0,2) circle (7pt);%\draw[thick] (0,3) circle (11pt);
\fill (1,8) circle (7pt);
\fill (2,0) circle (7pt);%\draw[thick] (2,5) circle (11pt);
\fill (3,1) circle (7pt);%\draw[thick] (3,0) circle (11pt);
\fill (4,6) circle (7pt);%\draw[thick] (4,2) circle (11pt);
\fill (5,5) circle (7pt);
\fill (6,7) circle (7pt);
\fill (7,4) circle (7pt);
\fill (8,3) circle (7pt);
\draw[dashed,thick, red] (5.5,-1)--(5.5,9);
%\draw[dashed,thick, blue] (-1,2.5)--(9,2.5);
\end{tikzpicture}$$

%Note that for any $w\in K_L(r,n)$ we have that $n\in \Left_r(w)$.

The set $S(i)$ can be thought of as the set of permutations in $K_L(r,n)$ for which $i$ is the smallest value appearing right of position $r$.  Note that, since $|\Left_r(w)|=r$, $S(i)$ is empty for all $i>r$ .  Since every permutation has some smallest value appearing right of position $r$, we get
\[K_L(r,n)=\bigsqcup_{i=1}^{r} S(i).\]
Fix $1\leq i<r$ and suppose that $w\in S(i)$.  Then the permutation matrix of $w$ can be divided into regions $A, B, C,D$ as follows:
$$
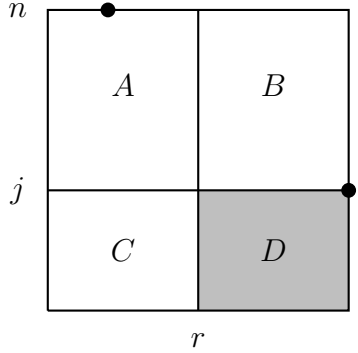
\begin{tikzpicture}[scale=0.4]
%\fill[lightgray] (0,10) rectangle (5,5);
\fill[lightgray] (10,0) rectangle (5,4);
\draw[thick] (0,0) -- (0,10) -- (10,10) -- (10,0) -- (0,0);
\draw[thick] (0,4) -- (10,4);
\draw[thick] (5,0) -- (5,10);
\fill (2,10) circle (7pt);
\fill (10,4) circle (7pt);
\draw (2.5,7.5) node {$A$};\draw (7.5,7.5) node {$B$}; \draw (2.5,2) node {$C$};
\draw (7.5,2) node {$D$};
\draw (-1,4) node {$j$};\draw (-1,10) node {$n$};
\draw (5,-1) node {$r$};
\end{tikzpicture}$$

Given a node $(k,w(k))$ in the permutation matrix of $w$, we will refer to $k$ as the position of the node and $w(k)$ as the value of the node.  Since $k\in \Left_r(w)$ for all $k<i$, we must have that region $D$ is empty (in other words, no nodes of the form $(k,w(k))$ appear in region $D$).  Since region $D$ is empty, the nodes in region $C$ can have any relative order with $w$ still avoiding $3|12$ and $23|1$ with respect to position $r$.  Moreover, the set of values of the nodes in region $C$ are exactly $\{1,2,\ldots, i-1\}$.  Observe that there are $(r)_{i-1}$ ways to assign $\{1,2,\ldots,i-1\}$ to the values $\{w(1),\ldots,w(r)\}$ in region $C$.  Any such assignment determines the positions of the nodes found in region $C$ and hence, it will also determine the positions of the nodes found in region $A$ (but not the values).  Next we determine the number of ways to select values in region $A$.  First note that  $n\in\Left_r(w)$ and $i\in\Right_r(w)$, and hence $n$ must the value of some node in region $A$, while $i$ is a value of some node in region $B$.  This leaves $\displaystyle\binom{n-i-1}{r-i}$ ways to select the remaining values for the nodes in region $A$.  Furthermore, since region $D$ is empty, any such choice uniquely determines the values of the nodes in region $B$.  Since $n\in\Left_r(w)$ and $w$ avoids $3|12$ with respect to position $r$, the values of the nodes in region $B$ must be decreasing.  Similarly, since $i\in\Right_r(w)$ and $w$ avoids $23|1$ with respect to position $r$, the values of the nodes in region $A$ are also decreasing.  The fact that the nodes in both regions $A$ and $B$ must be arranged in decreasing order implies
\[|S(i)|=\binom{n-i-1}{r-i}\cdot (r)_{i-1}\]
and thus
\[|K_L(r,n)|=\sum_{i=1}^{r}\binom{n-i-1}{r-i}\cdot (r)_{i-1}.\]
\end{proof}

%(note that this implies $w(n)=i$)

\begin{proof}[Proof of Theorem \ref{T:main1}]
Fix $1\leq r\leq n$ and observe that $|K_L(r,r)|=k(r,r)=r!$.  Corollary \ref{C:K2K_L} and Proposition \ref{P:K_L_formula} imply that
\begin{align*}
k(r,n)&=r!(n-r)!+\sum_{j=0}^{n-r-1}\left(\sum_{i=1}^{r}  \binom{n-j-i-1}{r-i}\cdot (r)_{i-1}\right)\cdot (n-r)_j\\
&=r!(n-r)!+\sum_{i=1}^{r}\sum_{j=1}^{n-r}\binom{n-i-j}{r-i}\cdot (r)_{i-1}\cdot(n-r)_{j-1}
\end{align*}

\end{proof}
\section{Generating functions}
In this section we prove Theorem \ref{T:main2}.  Our first step will be to isolate the summation term from the formula found in Theorem \ref{T:main1}.  We show that this summation term has a nice recursive structure.  For any $r,s\geq 0$, we define
\begin{equation}\label{D:asr}
a(r,s):=\frac{k(r,r+s)}{r!\cdot s!}-1.
\end{equation}
Note that $a(0,0)=0$ since $k(0,0)=1$.  We also have $a(r,s)=a(s,r)$ by the symmetry $k(r,r+s)=k(s,r+s)$.

\begin{lemma}\label{L:recursion}
The values $a(r,s)$ satisfy the recursion
\begin{equation}\label{Eq:recursion}
a(r,s)=a(r,s-1)+a(r-1,s)-a(r-1,s-1)+\binom{r+s-2}{r-1}\frac{1}{r!\, s!}
\end{equation}
with initial conditions $a(r,0)=a(0,s)=0$ for all $r,s\geq 0$.  
\end{lemma}
\begin{proof}
First note that $a(r,0)=a(0,s)=0$ since $k(r,r)=r!$.  Now if $r,s>0$, then Theorem \ref{T:main1} implies
\begin{align*}
a(r,s)&=\sum_{i=1}^{r}\sum_{j=1}^{s}\binom{r+s-i-j}{r-i} \frac{(r)_{i-1}(s)_{j-1}}{r!\, s!}\\
&=\sum_{i=1}^{r}\sum_{j=1}^{s}\binom{r+s-i-j}{r-i} \frac{1}{(r-i+1)!(s-j+1)!}.
\end{align*}
Reversing the ordering in the summations gives
\[a(r,s)=\sum_{i=0}^{r-1}\sum_{j=0}^{s-1}\binom{i+j}{i} \frac{1}{(i+1)! (j+1)!}.\]
Equation \eqref{Eq:recursion} immediately follows from this formula.
\end{proof}
Define the ordinary generating function
\[\mathcal{A}(x,y):=\sum_{r,s=0}^{\infty} a(r,s)\, x^r y^s.\]
Lemma \ref{L:recursion} implies
\begin{equation}\label{Eq:gen_func1}
\mathcal{A}(x,y)=\frac{\mathcal{L}(x,y)}{1-x-y+xy}
\end{equation}
where
\[\mathcal{L}(x,y):=\sum_{r,s=0}^{\infty}\binom{r+s-2}{r-1}\frac{x^r y^s}{r!\, s!}.\]
For the second step, we study the function $\mathcal{L}(x,y)$.  Taking the mixed partial derivative of $\mathcal{L}(x,y)$ yields a bivariate exponential generating function for binomial coefficients:
\begin{equation}\label{Eq:gen_func2}
\partial_{xy}(\mathcal{L}(x,y))=\sum_{r,s=0}^{\infty} \binom{r+s}{r}\frac{x^r y^s}{r!\, s!}.
\end{equation}
Let \[I_0(z)=\frac{1}{\pi}\int_0^\pi e^{z\cos\theta}d\theta\] denote the modified Bessel function of first kind.  It is well known that $I_0(z)$ has generating function

\[I_0(z)=\sum_{m=0}^\infty \frac{1}{(m!)^2}\left(\frac{z^2}{4}\right)^m.\]

\begin{proposition}\label{P:binom_gen_fun}
The exponential generating function for binomial coefficients is given by
\begin{equation}\label{Eq:Bessel}
\sum_{r,s=0}^{\infty} \binom{r+s}{r}\frac{x^r y^s}{r!\, s!}=e^{x+y}\cdot I_0(2\sqrt{xy}).
\end{equation}
\end{proposition}

\begin{proof}
Expanding the right hand side of Equation \eqref{Eq:Bessel} gives
\begin{align*}
e^{x+y}\cdot I_0(2\sqrt{xy})&=\left(\sum_{a,b=0}^{\infty}\frac{x^ay^b}{a!\, b!}\right)\left(\sum_{m=0}^\infty \frac{(xy)^m}{(m!)^2}\right)\\
%&=\sum_{a,b,m=0}^{\infty}\frac{x^{a+m}y^{b+m}}{a!b!m!m!}\\
&=\sum_{a,b,m=0}^{\infty}\binom{a+m}{m}\binom{b+m}{b}\frac{x^{a+m}y^{b+m}}{(a+m)!(b+m)!}.
\end{align*}
If we re-index the summands by setting $r=a+m$ and $s=b+m$ and apply Vandermonde's identity on binomials, we get
\begin{align*}
e^{x+y}\cdot I_0(2\sqrt{xy})&=\sum_{r,s=0}^{\infty}\left(\sum_{m=0}^{\min(r,s)}\binom{r}{m}\binom{s}{s-m}\right)\frac{x^{r}y^{s}}{r!\, s!}\\
&=\sum_{r,s=0}^{\infty}\binom{r+s}{s}\frac{x^{r}y^{s}}{r!\, s!}.
\end{align*}

\end{proof}

\begin{proof}[Proof of Theorem \ref{T:main2}]
By the definition given in Equation \eqref{D:asr}, we have
\begin{align*}
\mathcal{K}(x,y)&=\sum_{r,s=0}^{\infty}k(r,r+s) \frac{x^{r}y^{s}}{r!\, s!}\\
&=\sum_{r,s=0}^{\infty}(a(r,s)+1)\, x^{r}y^{s}\\
%&=\sum_{r,s=0}^{\infty}a(r,s)\, x^{r}y^{s}+\sum_{r,s=0}^{\infty}x^{r}y^{s}\\
&=\mathcal{A}(x,y)+\frac{1}{1-x-y+xy}.
\end{align*}
Equation \eqref{Eq:gen_func1} implies
\[\mathcal{K}(x,y)=\frac{\mathcal{L}(x,y)+1}{1-x-y+xy}.\]
The theorem now follows from Proposition \ref{P:binom_gen_fun} and Equation \eqref{Eq:gen_func2}.
\end{proof}
\bibliographystyle{amsalpha}
\bibliography{split_bib.bib}
\end{document}